\documentclass[12pt]{amsart}
\usepackage{amsmath}
\usepackage{amsthm}
\usepackage{latexsym}
\usepackage{amssymb}
\usepackage{mathrsfs}
\DeclareMathOperator{\Char}{Char}

\DeclareMathOperator{\re}{Re}
\DeclareMathOperator{\im}{Im}

\DeclareMathOperator{\id}{Id}

\DeclareMathOperator{\dist}{dist}

\DeclareMathOperator{\ord}{ord}

\newtheorem{theorem}{Theorem}
\newtheorem{proposition}{Proposition}

\newtheorem{corollary}{Corollary}
\newtheorem{definition}{Definition}
\newtheorem{example}{Example}
\newtheorem{remark}{Remark}

\def\jpopn#1#2{%
  \mathopen{%
    \setbox0=\hbox{$#1\langle$}%
    \setbox2=\hbox{%
            {\hbox{$#1\langle$}}%
            \kern -.6\wd0\box0%
    }%
            \box2%
  }%
}

\def\jpcls#1#2{%
  \mathclose{%
    \setbox0=\hbox{$#1\rangle$}%
    \setbox2=\hbox{%
            {\hbox{$#1\rangle$}}%
            \kern -.6\wd0\box0%
    }%
            \box2%
  }%
}

\def\tp{\ {}^{t}\kern -3pt}

\renewcommand{\thetheorem}{\thesection.\arabic{theorem}}
\renewcommand{\theproposition}{\thesection.\arabic{proposition}}
\renewcommand{\thelemma}{\thesection.\arabic{lemma}}
\renewcommand{\thedefinition}{\thesection.\arabic{definition}}
\renewcommand{\thecorollary}{\thesection.\arabic{corollary}}
\renewcommand{\theequation}{\thesection.\arabic{equation}}
\renewcommand{\theremark}{\thesection.\arabic{remark}}
\renewcommand{\theconjecture}{\thesection.\arabic{conjecture}}
\begin{document}
%
\def\A {{\mathcal{A}}}
\def\D {{\mathcal{D}}}
\def\R {{\mathbb{R}}}
\def\N {{\mathbb{N}}}
\def\C {{\mathbb{C}}}
\def\Z {{\mathbb{Z}}}
\def\Q {{\mathbb{Q}}}
\def\phi{\varphi}
\def\epsilon{\varepsilon}
\def\kappa{\varkappa}
%
%
%
%

\title[Perturbed Sums of Squares]{Analytic and Gevrey Hypoellipticity for Perturbed
  Sums of Squares Operators} 
\author{Antonio Bove}
\address{Dipartimento di Matematica, Universit\`a
di Bologna, Piazza di Porta San Donato 5, Bologna
Italy}
\email{bove@bo.infn.it}
\author{Gregorio Chinni}
\address{Dipartimento di Matematica, Universit\`a
di Bologna, Piazza di Porta San Donato 5, Bologna
Italy}
\email{gregorio.chinni@gmail.com}
\date{\today}

\begin{abstract}
We prove a couple of results concerning pseudodifferential
perturbations of differential operators being sums of squares of
vector fields and satisfying H\"ormander's condition. The first
is on the minimal Gevrey regularity: if a sum of squares with analytic
coefficients is perturbed with a pseudodifferential operator of order
strictly less than its subelliptic index it still has the Gevrey
minimal regularity. We also prove a statement concerning real analytic
hypoellipticity for the same type of pseudodifferential perturbations,
provided the operator satisfies to some extra conditions (see Theorem
\ref{th:2} below) that ensure the analytic hypoellipticity. 
\end{abstract}
\maketitle
%

\section{Introduction and Statement of the Result}
\setcounter{equation}{0}
\setcounter{theorem}{0}
\setcounter{proposition}{0}
\setcounter{lemma}{0}
\setcounter{corollary}{0}
\setcounter{definition}{0}
\setcounter{remark}{0}

Let $ X_{j}(x, D) $, $ j = 1, \ldots, N $, $ N \in \N $, be vector fields
defined in an open subset of $ U \subset \R^{n} $. We may suppose that
the origin belongs to $ U $ and that the vector fields have analytic
coefficients defined in $ U $. Let
\begin{equation}
\label{eq:P}
P (x, D) = \sum_{j=1}^{N} X_{j}(x, D)^{2},
\end{equation}
and assume that the vector fields satisfy the H\"ormander's condition:
\begin{itemize}
\item[(H) ]{}
The Lie algebra generated by the vector fields and their commutators
has dimension $ n $, equal to the dimension of the ambient space.
\end{itemize}
H\"ormander proved in \cite{hormander67} that (H) is sufficient for $
C^{\infty} $ hypoellipticity.

The operator $ P $ satisfies the \textit{a priori} estimate
\begin{equation}
\label{eq:Cinftysubell}
\| u \|_{1/r}^{2} + \sum_{j=1}^{N} \|X_{j} u \|_{0}^{2} \leq C \left(
  |\langle P u , u \rangle | + \| u\|^{2}_{0} \right),
\end{equation}
which we call, for the sake of brevity, the ``subelliptic estimate.''
Here $ u \in C_{0}^{\infty}(U) $, $ \| \cdot \|_{0}  $ denotes the
norm in $ L^{2}(U) $ and $ \| \cdot 
\|_{s} $ the Sobolev norm of order $ s $ in $ U $. 
Since the vector fields satisfy condition (H), we denoted by $ r $
the length of the iterated commutator such that the vector fields,
their commutators, their triple commutators etcetera up to the
commutators of length $ r $ generate a Lie algebra of dimension equal
to that of the ambient space. 

The above estimate was proved first by H\"ormander in
\cite{hormander67} for a Sobolev norm of order $ r^{-1} + \epsilon $
and up to order $ r^{-1} $ subsequently by Rothschild and Stein
(\cite{rs}) as well as in a pseudodifferential context by Bolley,
Camus and Nourrigat in \cite{bcn}.

Basically using \eqref{eq:Cinftysubell} Derridj and Zuily proved in
\cite{dz} that any operator of the form \eqref{eq:P} is Gevrey
hypoelliptic of order $ r $, i.e. that if $ u $ is a distribution on
an open set $ U $ such that $ P u \in G^{r}(U) $ then $ u \in G^{r}(U)
$. In \cite{abc} a microlocal version of this has been proved and we
refer to subsection 2.4 for more details.

The purpose of this note is to study the following problem: when the
hypoellipticity properties of the operator $ P $ are preserved if we
are willing to perturb it with an analytic pseudifferential operator? 

It is known (see \cite{horm3}, Theorems 22.4.14 as well as 22.4.15)
that if we perturb a sum of squares with an arbitrary first order
operator we may obtain a non hypoelliptic operator. For instance if we
consider $ P (x, D) = D_{1}^{2} + x_{1}^{2} D_{2}^{2} $ in two
variables and perturb it with a first order operator: $ \tilde{P} (x,
D) = D_{1}^{2} + x_{1}^{2} D_{2}^{2} + \alpha D_{2} $, we have a non
hypoelliptic operator if $ \alpha = \pm 1 $ or if $ \alpha $ is a
function assuming those values at the point of interest.

In a sort of converse direction Stein, in \cite{stein}, proved that if
we consider Kohn's Laplacian, $ \Box_{b} $, which is neither
hypoelliptic nor analytic hypoelliptic, and perturb it with a non zero
complex number, $ \Box_{b} + \alpha $, $ \alpha \in \C\setminus\{0\}
$, we obtain an operator being both hypoelliptic and analytic
hypoelliptic. 

For further details on (first order) differential perturbations we
refer to the papers \cite{gt} and \cite{mm}. For a pseudodifferential
perturbation we give, in the Appendix, a very brief account showing
that the order of the perturbation does matter lest we have to impose
extra conditions on the perturbing symbol. 

These facts suggest that,
if no other conditions are to be imposed on the perturbing operator,
its order has to be strictly less than the subelliptic index of the
sum of squares.

\bigskip

Before stating our result we need some notation.

Write $ \{X_{i} , X_{j}\} $ for the Poisson bracket of the symbols
of the vector fields $ X_{i} $, $ X_{j} $:
$$ 
\{X_{i} , X_{j}\} (x, \xi) = \sum_{\ell=1}^{n} \left( \frac{\partial
    X_{i}}{\partial \xi_{\ell}} \frac{\partial X_{j}
  }{\partial x_{\ell}} - 
  \frac{\partial X_{j}}{\partial \xi_{\ell}} \frac{\partial X_{i}
  }{\partial x_{\ell}} \right) (x, \xi) .
$$
\begin{definition}
\label{def:nu}
Fix a point $ (x_{0}, \xi_{0}) \in \Char(P) $. Consider all the
iterated Poisson brackets $ \{X_{i}, X_{j}\} $, $ \{ \{X_{i}, X_{j}\},
X_{k} \}$ etcetera.

We define $ \nu (x_{0}, \xi_{0})  $ as the length
of the shortest iterated Poisson bracket of the symbols of the vector
fields which is non zero at $ (x_{0}, \xi_{0}) $. 
\end{definition}
Now we have
\begin{theorem}
\label{th:1}
Let $ P $ be as in \eqref{eq:P} and denote by $  Q(x, D) $ an
analytic pseudodifferential operator defined in a conical neighborhood
of the point $ (x_{0}, \xi_{0}) \in \Char(P) $---the characteristic
variety of $ P $. If 
$$ 
\ord(Q) < 2/\nu(x_{0}, \xi_{0})  
$$ 
then $ P + Q $ is $ G^{\nu(x_{0}, \xi_{0})} $ hypoelliptic at $ (x_{0}, \xi_{0})  $.
\end{theorem}
A few remarks are in order.

\begin{itemize}
\item[(a) ]{}
Definition \ref{def:nu} as well as the regularity
obtained in Theorem \ref{th:1} microlocal. We say that an operator $ Q
$ is $ G^{s} $ hypoelliptic at $ (x_{0}, \xi_{0})  $ if $ (x_{0},
\xi_{0}) \not\in WF_{s}(u)  $ provided $ (x_{0}, \xi_{0}) \not \in
WF_{s}(Qu) $.
\item[(b) ]{}
We stated Theorem \ref{th:1} in the case of analytic coefficients, for
the sake of simplicity. Actually one might assume some Gevrey
regularity like we do in the following corollary. 
\end{itemize}
\begin{corollary}
\label{cor:1}
Let $ V $ denote a neighborhood of the point $ x_{0} $ and 
$$ 
r = \sup_{x \in V, |\xi| = 1} \nu (x, \xi).
$$
Let moreover $ P $ be as
above with $ G^{r} $ coefficients defined in $ V $ and $ Q \in
OPS^{m}_{r}(V) $ be a $ G^{r} $ pseudodifferential operator of order $
m < 2/r $. Then $ P + Q $ is $ G^{r} $ hypoelliptic at $ x_{0} $.
\end{corollary}

\bigskip

A perturbation result for the analytic case can also be proved using
the same ideas as for Theorem \ref{th:1}.

We make the following assumptions on the operator $ P $ in
\eqref{eq:P}:
\begin{itemize}
\item[(1) ]{}
Let $ U \times \Gamma $ be a conic neighborhood of $ (x_{0}, \xi_{0})
$. There exists a real analytic function, $ h (x, \xi) $, $ h \colon U
\times \Gamma \rightarrow [0, +\infty [ $ such that $ h (x_{0},
\xi_{0}) = 0 $ and $ h (x, \xi) > 0 $ in $ U \times \Gamma \setminus
\{ (x_{0}, \xi_{0}) \} $. 
\item[(2) ]{}
There exist real analytic functions $ \alpha_{jk}(x, \xi) $ defined in
$ U \times \Gamma $, such that
\begin{equation}
\label{eq:riprod}
\{ h (x, \xi) , X_{j}(x, \xi) \} = \sum_{\ell=1}^{N} \alpha_{j\ell}(x,
\xi) X_{\ell}(x, \xi) , 
\end{equation}
for $ j = 1, \ldots, N $.
\end{itemize}
In \cite{AB} it was proved that if $ P $, defined as in \eqref{eq:P},
satisfies (1), (2) then $ P $ is analytic hypoelliptic at $ (x_{0},
\xi_{0}) $. 
\begin{theorem}
\label{th:2}
Let $ P $ be as in \eqref{eq:P} and assume that (1) and (2) above are
satisfied. Let $ Q $ be a real analytic pseudodifferential operator
of order strictly less than $ 2/\nu(x_{0}, \xi_{0}) $, then $ P + Q
$ is analytic hypoelliptic at $ (x_{0}, \xi_{0}) $.
\end{theorem}
We point out that the \textit{ideal} statement of the above theorem
would be one deducing analytic hypoellipticity of the perturbation
from the analytic hypoellipticity of the operator, without any
assumption but the order of the perturbation. Unfortunately this seems
a much more difficult result to prove and it has been proved in the
global case, for some classes of operators, by Chinni and Cordaro,
\cite{cc}, and by Braun Rodrigues, Chinni, Cordaro and Jahnke,
\cite{bccj}. 

Finally we say a few words about the method of proof. It consists in
using the FBI transform and the subelliptic inequality on the FBI
side obtained in \cite{abc}. To do that we use a deformation technique
of the Lagrangean associated to the FBI proposed by Grigis and
Sj\"ostrand in \cite{gs}.

\section{Background on FBI and Sums of Squares}
\setcounter{equation}{0}
\setcounter{theorem}{0}
\setcounter{proposition}{0}
\setcounter{lemma}{0}
\setcounter{corollary}{0}
\setcounter{definition}{0}
\setcounter{remark}{0}

We are going to use a pseudodifferential and FIO (Fourier Integral
Operators) calculus introduced by Grigis and Sj\"ostrand in the paper
\cite{gs}. We recall below the main definitions
and properties to make this paper self-consistent and readable. For
further details we refer to the paper \cite{gs} and to the lecture
notes \cite{Sj-reson}.

\subsection{The FBI Transform}
\renewcommand{\thetheorem}{\thesubsection.\arabic{theorem}}
\renewcommand{\theproposition}{\thesubsection.\arabic{proposition}}
\renewcommand{\thelemma}{\thesubsection.\arabic{lemma}}
\renewcommand{\thedefinition}{\thesubsection.\arabic{definition}}
\renewcommand{\thecorollary}{\thesubsection.\arabic{corollary}}
\renewcommand{\theequation}{\thesubsection.\arabic{equation}}
\renewcommand{\theremark}{\thesubsection.\arabic{remark}}

We define the FBI transform of a temperate distribution $ u $ as
$$ 
Tu(x, \lambda) = \int_{\R^{n}} e^{i \lambda \phi(x, y)} u(y) dy,
$$
where $ \lambda \geq 1 $ is a large parameter, $ \phi $ is a
holomorphic function such that $ \det \partial_{x}\partial_{y}\phi
\neq 0 $, $ \im \partial_{y}^{2}\phi > 0 $.

Here $ \partial_{x} $ denotes the complex derivative with respect to
the complex variable $ x $.
\begin{example}
\label{fbiclassic}
A typical phase function may be $ \phi(x, y) = \frac{i}{2} (x-y)^{2} $.
\end{example}
To the phase $ \phi $ there corresponds a weight function $ \Phi(x) $,
defined as
$$ 
\Phi(x) = \sup_{y \in \R^{n}} - \im \phi(x, y), \qquad x \in \C^{n}.
$$
\bigskip

We may take a slightly different perspective. Let us consider $
(x_{0}, \xi_{0})$  $\in \C^{2n} $ and a real valued real analytic
function $ \Phi(x) $ defined near $ x_{0} $, such that $ \Phi $ is
strictly plurisubharmonic and
$$ 
\frac{2}{i}\ \partial_{x}\Phi(x_{0}) =\xi_{0}.
$$
Denote by $ \psi(x, y) $ the holomorphic function defined near $
(x_{0}, \bar{x}_{0}) $ by 
\begin{equation}
\label{psi}
\psi(x, \bar{x}) = \Phi(x).
\end{equation}
Because of the plurisubharmonicity of $ \Phi $, we have 
\begin{equation}
\label{psixy}
\det \partial_{x} \partial_{y} \psi \neq 0 
\end{equation}
and 
\begin{equation}
\label{repsi}
\re \psi(x, \bar{y}) - \frac{1}{2}\left [ \Phi(x) + \Phi(y) \right]
\sim - |x - y |^{2}.
\end{equation}
To end this Section we recall the definition of $ s $--Gevrey wave
front set of a distribution.
\begin{definition}
\label{def:swf}
Let $ (x_{0}, \xi_{0}) \in U \subset T^{*}\R^{n}\setminus 0 $. We say
that $ (x_{0}, \xi_{0}) \notin WF_{s}(u) $ if there exist a
neighborhood $ \Omega $ of $ x_{0} - i \xi_{0} \in \C^{n} $ and
positive constants $ C_{1} $, $ C_{2} $ such that
$$ 
| e^{-\lambda \Phi_{0}(x)} Tu(x, \lambda) | \leq C_{1} e^{- \lambda^{1/s}/C_{2}},
$$
for every $ x \in \Omega $. Here $ T $ denotes the classical FBI
transform, i.e. that using the phase function of Example \ref{fbiclassic}.
\end{definition}
\subsection{Pseudodifferential Operators}
\setcounter{equation}{0}
\setcounter{theorem}{0}
\setcounter{remark}{0}

Let $ \lambda \geq 1 $ be a large positive parameter. We write 
$$ 
\tilde{D} = \frac{1}{\lambda}  D, \qquad  D = \frac{1}{i} \partial.
$$
Denote by $ q(x, \xi, \lambda) $ an analytic classical symbol and
by $ Q(x, \tilde{D}, \lambda) $ the formal classical
pseudodifferential operator associated to $ q $.

Using ``Kuranishi's trick'' one may represent  $ Q(x, \tilde{D},
\lambda) $ as
\begin{equation}
\label{qk}
Q u(x, \lambda) = \left( \frac{\lambda}{2 i\pi }\right )^{n} \int
  e^{2\lambda (\psi(x, \theta) - \psi(y, \theta))} \tilde{q}(x, \theta,
  \lambda) u(y) dy d\theta.
\end{equation}
Here $ \tilde{q} $ denotes the symbol of $ Q $ in the actual
representation. 

To realize the above operator we need a prescription for the
integration path.

This is accomplished by transforming the classical integration path
via the Kuranishi change of variables and eventually applying Stokes
theorem: 
\begin{equation}
\label{realization} 
Q^{\Omega}u(x, \lambda) = \left( \frac{\lambda}{\pi}\right)^{n}
\int_{\Omega} e^{2 \lambda \psi(x, \bar{y})} \tilde{q}(x, \bar{y}, \lambda)
u(y) e^{-2 \lambda\Phi(y)} L(dy),
\end{equation}
where $ L(dy) = (2i)^{-n} dy \wedge d\bar{y} $, the integration path
is $ \theta = \bar{y} $ and $ \Omega $ is a small neighborhood of $
(x_{0}, \bar{x}_{0}) $. We remark that $Q^{\Omega}u (x)$ is an holomorphic
function of $x$. 

\begin{definition}
\label{def:Hphi}
Let $ \Omega $ be an open subset of $ \C^{n} $.
We denote by $ H_{\Phi}(\Omega) $ the space of all holomorphic functions $ u(x,
\lambda) $ such that for every $ \epsilon > 0 $ and for every compact
$ K \subset\!\subset \Omega $ there exists a constant $ C > 0 $ such that
$$ 
| u(x, \lambda) | \leq C e^{\lambda (\Phi(x) + \epsilon)},
$$
for $ x \in K $ and $ \lambda \geq 1 $.
\end{definition}
\begin{remark}
\label{rem:1}
If $ \tilde{q} $ is a classical symbol of order zero, $ Q^{\Omega} $
is uniformly bounded as $ \lambda \rightarrow +\infty $, from $
H_{\Phi}(\Omega) $ into itself.
\end{remark}
\begin{remark}
\label{rem:2}
If the principal symbol is real, $ Q^{\Omega} $ is formally self
adjoint in $L^{2}(\Omega,$  $ e^{ -2 \lambda \Phi})$.
\end{remark}
\begin{remark}
\label{rem:3}
The definition \eqref{qk} of (the realization of) a pseudodifferential
operator on an open subset $ \Omega $ of $ \C^{n} $ is not the
classical one. Via the Kuranishi trick it can be reduced to the
classical definition. On the other hand using the function $ \psi $
allows us to use a weight function not explicitly related to an FBI
phase. This is useful since in the proof we deform the Lagrangean $
\Lambda_{\Phi_{0}} $, corresponding e.g. to the classical FBI phase,
and obtain a \textit{deformed} weight function which is useful in the
a priori estimate.
\end{remark}

For future reference we also recall that the identity operator can be
realized as
\begin{equation}
\label{idomega}
I^{\Omega}u(x, \lambda) = \left( \frac{\lambda}{\pi}\right)^{n}
\int_{\Omega} e^{2 \lambda \psi(x, \bar{y})} i(x, \bar{y}, \lambda)
e^{-2\lambda \Phi(y)} u (y, \lambda) L(dy),
\end{equation}
for a suitable analytic classical symbol $ i(x, \xi, \lambda) $.
Moreover we have the following estimate (see \cite{gs} and
\cite{Sj-Ast})
\begin{equation}
\label{errorest}
\| I^{\Omega} u - u \|_{\Phi - d^{2}/C} \leq C' \| u \|_{\Phi + d^{2}/C},
\end{equation}
for suitable positive constants $ C $ and $ C' $. Here we denoted by
\begin{equation}
\label{d}
d(x) = \dist(x, \complement \Omega),
\end{equation}
the distance of $ x $ to the boundary of $ \Omega $, and by
\begin{equation}
\label{eq:phinorm}
\| u \|_{\Phi}^{2} = \int_{\Omega} e^{-2\lambda \Phi(x)} |u(x) |^{2} L(dx).
\end{equation}

\subsection{Some Pseudodifferential Calculus}
\setcounter{equation}{0}
\setcounter{theorem}{0}
\setcounter{proposition}{0}  
\setcounter{lemma}{0}
\setcounter{corollary}{0} 
\setcounter{definition}{0}
\setcounter{remark}{0}

We start with a proposition on the composition of two
pseudodifferential operators.
\begin{proposition}[\cite{gs}]
\label{composition}
Let $ Q_{1} $ and $ Q_{2} $ be of order zero. Then they can be
composed and
$$ 
Q_{1}^{\Omega} \circ Q_{2}^{\Omega} = (Q_{1} \circ Q_{2})^{\Omega} + R^{\Omega},
$$
where $ R^{\Omega} $ is an error term, i.e. an operator whose norm is
$ \mathscr{O}(1) $ as an operator from $ H_{\Phi + (1/C) d^{2}} $ to $
H_{\Phi - (1/C) d^{2}} $ 
\end{proposition}
We shall need also a lower bound for an elliptic operator of order
zero.
\begin{proposition}[\cite{abc}]
\label{lowerbound}
Let $ Q $ a zero order pseudodifferential operator defined on $ \Omega
$ as above. Assume further that its principal symbol $ q_{0}(x, \xi,
\lambda) $ satisfies
$$ 
| {q_{0}}_{|_{\Lambda_{\Phi} \cap \pi^{-1}(\Omega)}} | \geq c_{0} > 0.
$$
Here $ \pi $ denotes the projection onto the first factor in $
\C^{n}_{x} \times \C^{n}_{\xi} $. Then
\begin{equation}
\label{lb}
\| u \|_{\tilde{\Phi}} + \| Q^{\Omega} u \|_{\Phi} \geq C \| u \|_{\Phi},
\end{equation}
where 
\begin{equation}
\label{phitilde}
\tilde{\Phi}(x) = \Phi(x) + \frac{1}{C} d^{2}(x),
\end{equation}
and $ d $ has been defined in (\ref{d}).
\end{proposition}
\begin{proof}
We have
\begin{multline*} 
Q^{\Omega} u(x, \lambda) - {q_{0}}_{|_{\Lambda_{\Phi}}}(x, \lambda)
I^{\Omega} u(x, \lambda) \\
= \left( \frac{\lambda}{\pi} \right)^{n} \int_{\Omega}
e^{2\lambda \psi(x, \bar{y})} \left [ q(x, \bar{y}, \lambda) -
  {q_{0}}_{|_{\Lambda_{\Phi}}} (x, \lambda) i(x, \bar{y}, \lambda)
\right]
\\
\cdot e^{-2\lambda \Phi(y)} u(y) L(dy).
\end{multline*}
The absolute value of the term in square brackets may be estimated by
$ C (|x -y|+\lambda^{-1})$. Then
\begin{multline*}
\| Q^{\Omega} u - {q_{0}}_{|_{\Lambda_{\Phi}}} I^{\Omega} u
\|_{\Phi}^{2}
\leq C \lambda^{-2} \| u \|_{\Phi}^{2}
\\
+C  \int_{\Omega} \left | \left(\frac{\lambda}{\pi}\right)^{n}
  \int_{\Omega} 
e^{-\lambda \Phi(x) + 2 \lambda \psi(x, \bar{y}) - \lambda \Phi(y)} |x
- y| e^{-\lambda \Phi(y)} u(y) L(dy) \right|^{2} L(dx) \\
\leq
C \left(\frac{\lambda}{\pi}\right)^{2n} \int_{\Omega} 
\left( \int_{\Omega} e^{-\lambda/C |x-y|^{2}} |x-y| L(dy) \right)
\\
\cdot
\left( \int_{\Omega} e^{-\lambda/C |x-y|^{2}} |x-y| e^{-2\lambda
    \Phi(y)} |u(y)|^{2}  L(dy) \right) L(dx)+C\lambda^{-2} \| u \|_{\Phi}^{2}
\\
\leq C \lambda^{-1} \| u \|_{\Phi}^{2}.
\end{multline*}
Using (\ref{errorest}) we may conclude that
\begin{multline*}
\| Q^{\Omega} u \|_{\Phi} \geq \| {q_{0}}_{|_{\Lambda_{\Phi}}}
I^{\Omega} u \|_{\Phi} - C \lambda^{-1/2} \| u \|_{\Phi} 
\\
\geq 
\| {q_{0}}_{|_{\Lambda_{\Phi}}}  u \|_{\Phi} - \| {q_{0}}_{|_{\Lambda_{\Phi}}}
(I^{\Omega} - 1) u \|_{\Phi} -  C \lambda^{-1/2} \| u \|_{\Phi} 
\\
\geq
c_{0} \| u \|_{\Phi} - C \| u \|_{\tilde{\Phi}} - C \lambda^{-1/2} \|
u \|_{\Phi}.
\end{multline*}
This proves the assertion.
\end{proof}

\subsection{An \textit{a priori} Estimate for Sums of Squares}
\setcounter{equation}{0}
\setcounter{theorem}{0}
\setcounter{proposition}{0}  
\setcounter{lemma}{0}
\setcounter{corollary}{0} 
\setcounter{definition}{0}
\setcounter{remark}{0}

Consider now the vector fields $ X_{j} $ defined in Section
1. Following \cite{abc} we state the FBI version of the
estimate \eqref{eq:Cinftysubell}.

\begin{theorem}
\label{apriorifriendly}
Let $ P^{\Omega} $ be the $ \Omega $-realization of $ P $ (see
equation \eqref{realization}.) Note that, arguing as in \cite{gs}
we have that
\begin{equation}
\label{eq:Preal}
P^{\Omega} = \sum_{j=1}^{N} (X_{j}^{\Omega})^{2} + \mathscr{O}(\lambda^{2}),
\end{equation}
where $\mathscr{O}(\lambda^{2}) $ is continuous from $
H_{\tilde{\Phi}} $ to $ H_{\Phi - (1/C)d^{2}} $ with norm boun\-ded by $
C' \lambda^{2} $, $ \tilde{\Phi} $ given by \eqref{phitilde}. 

Let $ \Omega_{1} \subset\!\subset \Omega $.
Then
\begin{equation}
\label{apriorifin}
\lambda^{\frac{2}{r}} \| u \|_{\Phi}^{2} +
\sum_{j=1}^{N} \| X_{j}^{\Omega} u\|_{\Phi}^{2} 
\leq C \left( \langle P^{\Omega}u , u\rangle_{\Phi} + \lambda^{\alpha}
  \|u\|_{\Phi, \Omega\setminus\Omega_{1}}^{2} \right),
\end{equation}
where $ \alpha $ is a positive integer, $ u \in L^{2}(\Omega, e^{-2
\Phi}L(dx)) $ and $ r = \nu((x_{0}, \xi_{0})) $. 
\end{theorem}

\section{Proof of Theorem \ref{th:1}}
\renewcommand{\thetheorem}{\thesection.\arabic{theorem}}
\renewcommand{\theproposition}{\thesection.\arabic{proposition}}
\renewcommand{\thelemma}{\thesection.\arabic{lemma}}
\renewcommand{\thedefinition}{\thesection.\arabic{definition}}
\renewcommand{\thecorollary}{\thesection.\arabic{corollary}}
\renewcommand{\theequation}{\thesection.\arabic{equation}}
\renewcommand{\theremark}{\thesection.\arabic{remark}}
\renewcommand{\theconjecture}{\thesection.\arabic{conjecture}}
\setcounter{equation}{0}
\setcounter{theorem}{0}
\setcounter{proposition}{0}
\setcounter{lemma}{0}
\setcounter{corollary}{0}
\setcounter{definition}{0}
\setcounter{remark}{0}

In order to prove Theorem \ref{th:1} we construct a deformation of $
\Lambda_{\Phi_{0}} $ following the ideas in \cite{gs} (see also
\cite{abc}.) 

Let us consider the ``sum of squares of vector fields'' operator $ P $
defined in \ref{eq:P}. Let $ (x_{0}, \xi_{0}) $ be a characteristic
point of $ P $ and let $ r = \nu(x_{0}, \xi_{0}) $.

We perform an FBI transform of the form
$$ 
Tu(x, \lambda) = \int_{\R^{n}} e^{ i\lambda \phi(x, y)} u(y) dy,
$$
where $ u $ is a compactly supported distribution and $ \phi(x, y) $ is
a phase function. Even though it does not really matter which phase
function we use, the classical phase function will be employed:
\begin{equation}
\label{eq:4.1}
\phi_{0}(x, y) = \frac{i}{2} (x - y)^{2}, \qquad x \in \C^{n}, y \in \R^{n}. 
\end{equation}
Let us denote by $ \Omega $ an open neighborhood of the point $
\pi_{x}\mathscr{H}_{T}(x_{0}, \xi_{0}) $ in $ \C^{n} $. Here $ \pi_{x}
$ denotes the space projection $ \pi_{x} \colon \C^{n}_{x} \times
\C^{n}_{\xi} \rightarrow \C^{n}_{x} $ and $ \mathscr{H}_{T} $ is the
complex canonical transformation associated to $ T $: 
$$ 
\mathscr{H}_{T} \colon \left\{ \left( y, - \frac{2}{i}
    \frac{\partial \Phi}{\partial y}\right) \right\} \rightarrow 
\left\{ \left( x, \frac{2}{i}
    \frac{\partial \Phi}{\partial x}\right) \right\},
$$
($ \Phi(x, y) = - \im \phi(x, y) $) i.e., in the classical case, once
we restrict to $ \R^{2n} $, 
$$ 
\mathscr{H}_{0} (y, \eta) = (y-i\eta, \eta), \quad (y, \eta) \in
\R^{2n}. 
$$
For the sake
of simplicity we denote by $ x_{0} \in \C^{n} $ the point
$\pi_{x}\mathscr{H}_{0}(x_{0}, \xi_{0}) $.

Let $ \Phi_{0}(x, y) = - \im \phi_{0}(x, y) = - \frac{1}{2} (x' -
y)^{2} + \frac{1}{2} x''^{2} $, where $ y \in \R^{n} $, $ x = x' + i
x'' \in \C^{n} $. We write also
$$ 
\Phi_{0}(x) = \text{c.v.}_{y \in \R^{n}} \Phi_{0}(x, y)
$$
(the critical value of $ \Phi_{0} $ w.r.t. $ y $.)

For $ \lambda \geq 1 $ let us consider a real analytic function
defined near the point $ \mathscr{H}_{0}(x_{0}, \xi_{0}) = (x_{0} - i
\xi_{0}, \xi_{0}) \in \Lambda_{\Phi_{0}} $, say $ h(x, \xi, \lambda) $. 
Solve, for small positive $ t $, the Hamilton-Jacobi problem
\begin{equation}
\label{eq:hj}
\left\{
\begin{matrix}
2 \dfrac{\partial \Phi}{\partial t}(t, x, \lambda) = h \left(x,
  \dfrac{2}{i} \dfrac{\partial \Phi}{\partial x}(t, x, \lambda), \lambda \right) \\
\Phi(0, x, \lambda) = \Phi_{0}(x) \hfill
\end{matrix}
\right . .
\end{equation}
This is easy to solve since $ h $ is real analytic. Set
$$
\Phi_{t}(x, \lambda) = \Phi(t, x, \lambda).
$$
We have
$$
\Lambda_{\Phi_{t}} = \exp\left(i t H_{h}\right) \Lambda_{\Phi_{0}}
$$
We choose the function $ h $ as
\begin{equation}
\label{eq:h}
h (x, \xi, \lambda) = \lambda^{-\frac{r-1}{r}} | x - x_{0} |^{2} \quad
\text{on } \  \Lambda_{\Phi_{0}}.
\end{equation}
Keeping in mind the definition of $ \Lambda_{\Phi_{0}} $, we have
that, as a function in $ \R^{2n} $
\begin{equation}
\label{eq:hLphi0}
h (x, \xi, \lambda) = \lambda^{-\frac{r-1}{r}} \left[ | x - x_{0} |^{2} + | \xi
  - \xi_{0} |^{2} \right].
\end{equation}
The function $ \Phi_{t} $ can be expanded as a power series in the
variable $ t $ using both equation \eqref{eq:hj} and the Fa\`a di
Bruno formula to obtain
\begin{equation}
\label{eq:Phit}
\Phi_{t}(x, \lambda) = \Phi_{0}(x) + \frac{t}{2}  \ h(\cdot, \cdot,
\lambda)_{\big |_{\Lambda_{\Phi_{0}}}} + \mathscr{O}(\lambda^{-1}) ,
\end{equation}
where $ h $ on $ \Lambda_{\Phi_{0}} $ is given by \eqref{eq:hLphi0}. 

Our purpose is to use the estimate \eqref{apriorifin} where the weight
function $ \Phi $ has been replaced by the weight $ \Phi_{t} $. This
is possible using the phase $ \psi_{t} $ in \eqref{qk} and realizing
the operator as in \eqref{realization}. Here $ \psi_{t} $ is defined
as the holomorphic extension of $ \psi_{t}(x, \bar{x}) = \Phi_{t}(x)
$. 

We need to restrict the symbol of both $ P $ and $ P + Q $ to $
\Lambda_{\Phi_{t}} $; denote by $ P^{t} $, $ Q^{t} $ the symbols of $
P $, $ Q $ restricted to $\Lambda_{\Phi_{t}} $. 

Noting that
\begin{multline*}
X_{j}^{2}(x, \frac{2}{i} \partial_{x}\Phi_{t}(x, \lambda), \lambda) =
X_{j}^{2}(x, \frac{2}{i} \partial_{x}\Phi_{0}(x), \lambda) 
\\
+ 2 t
X_{j}(x, \frac{2}{i} \partial_{x}\Phi_{0}(x), \lambda)
\langle \partial_{\xi} X_{j}(x,
\frac{2}{i} \partial_{x}\Phi_{0}(x),
\lambda), \frac{2}{i} \partial_{x}\partial_{t} {\Phi_{t}(x, \lambda)}_{\big
  |_{t=0}} \rangle 
\\
+ \mathscr{O}(t^{2} \lambda^{2/r}),  
\end{multline*}
We then deduce that
\begin{equation}
\label{lt}
P^{t}(x, \xi, \lambda) = \lambda^{2} P(x, \xi) + t R(x, \xi, \lambda) +
\mathscr{O}(t^{2} \lambda^{\frac{2}{r}}),
\end{equation}
where
$$
R(x, \xi, \lambda) = \lambda^{\frac{1}{r}} \sum_{j=1}^{N} a_{j}(x,
\xi, \lambda) X_{j}(x, \xi, \lambda).
$$
The analytic extension of $ P^{t} $ is the symbol appearing in the $
\Omega $-realization of $ P^{t} $, $ {P^{t}}^{\Omega} $. We point out
that the principal symbol of $ P^{t} $ satisfies the assumptions of
Theorem \ref{apriorifriendly} and, using the a priori inequality
(\ref{apriorifin}), we can deduce an estimate of the form (\ref{apriorifin})
for $ P^{t} $ in the $ H_{\Phi_{t}} $ spaces.

Denote by $ \theta $ the order of the pseudodifferential operator $ Q
$. We have
\begin{multline*}
\lambda^{\frac{2}{r}}\| u \|_{\Phi_{t}}^{2} + \sum_{j=1}^{N}
\|X_{j}^{\Omega} u\|^{2}_{\Phi_{t}} 
\\
 \leq C \left( |\langle ({P^{t}}^{\Omega} - t R^{\Omega} - \mathscr{O}(t^{2} \lambda^{\frac{2}{r}}))u
  , u\rangle_{\Phi_{t}}| + \lambda^{\alpha}
  \|u\|_{\Phi_{t}, \Omega\setminus\Omega_{1}}^{2} \right)
\\
=
C \left(| \langle ({P^{t}}^{\Omega} + {Q^{t}}^{\Omega} - t R^{\Omega} -
  \mathscr{O}(t^{2} \lambda^{\frac{2}{r}}) - {Q^{t}}^{\Omega})u
  , u\rangle_{\Phi_{t}}| \right . 
\\
\left .
+ \lambda^{\alpha}
  \|u\|_{\Phi_{t}, \Omega\setminus\Omega_{1}}^{2} \right)
\end{multline*}
The fourth term in the left hand side of the scalar product above is
easily absorbed on the left provided $ t $ is small enough. The fifth
term is also absorbed since, being $ Q $ of order $ \theta $, $ \|
{Q^{t}}^{\Omega} u \|_{\Phi_{t}, \Omega} \leq \lambda^{\theta} C \| u
\|_{\Phi_{t}, \Omega} $. 

Let us consider the third term in the scalar product above. 
By Proposition \ref{composition}, we have
$$ 
R^{\Omega} = \sum_{j=1}^{N} a_{j}^{\Omega}(x, \tilde{D}, \lambda)
X_{j}^{\Omega}(x, \tilde{D}, \lambda) + \mathscr{O}(\lambda),
$$
where $ \mathscr{O}(\lambda) $ denotes an operator from $ H_{\Phi_{t}
  + \frac{1}{C} d^{2}} $ to $ H_{\Phi_{t} - \frac{1}{C} d^{2}} $ whose
norm is bounded by $ C \lambda $. Hence
$$
t | \langle R^{\Omega}u, u\rangle_{\Phi_{t}} | 
\leq C t \Big(
\lambda^{\frac{2}{r}}\| u \|_{\Phi_{t}}^{2} + \sum_{j=1}^{N}
\|X_{j}^{\Omega} u\|^{2}_{\Phi_{t}} + \lambda^{2} \|u
\|^{2}_{\tilde{\Phi}_{t}} \Big).
$$
Hence we deduce that there exist a neighborhood $ \Omega_{0} $ of $ x_{0} $, a positive
number $ \delta $ and a positive integer $ \alpha $
such that, for every $ \Omega_{1} \subset\!\subset \Omega_{2}
\subset\!\subset\Omega \subset \Omega_{0} $, there exists a constant $ C > 0
$ such that, for $ 0 < t < \delta $, we have 
\begin{equation}
\label{scalprodt}
\lambda^{\frac{2}{r}}\| u \|_{\Phi_{t}, \Omega_{1}} 
 \leq C \left( \| {(P + Q)^{t}}^{\Omega} u\|_{\Phi_{t}, \Omega_{2}} +
   \lambda^{\alpha} \|u\|_{\Phi_{t}, \Omega\setminus\Omega_{1}} \right). 
\end{equation}
In other words Theorem \ref{apriorifriendly} holds for the perturbed
operator. 

Using \eqref{scalprodt} we may finish the proof of Theorem \ref{th:1}.

By assumption $\| {(P + Q)^{t}}^{\Omega} u\|_{\Phi_{t}, \Omega_{2}}
\leq C e^{- \lambda/C}  $, since $ \Lambda_{\Phi_{t}} $ is a small
perturbation of $ \Lambda_{\Phi_{0}} $ when $ t $ is small. 

By our choice of $ h $ (see \eqref{eq:h}) it is also straightforward
that $\|u\|_{\Phi_{t}, \Omega\setminus\Omega_{1}} \leq C e^{-
  \lambda^{1/r}/C}  $. Thus we obtain that
$$ 
\| u \|_{\Phi_{t}, \Omega_{1}} \leq C_{1} e^{- \lambda^{1/r} /C_{1}} .
$$
On the other hand $ \Phi_{t}(x, \lambda) = \Phi_{0}(x) + \frac{t}{2}
\ h(\cdot, \cdot, \lambda)_{\big |_{\Lambda_{\Phi_{0}}}} +
\mathscr{O}(\lambda^{-1}) $, so that, if we are close enough to the
base point on $ \Lambda_{\Phi_{0}} $, i.e. for $ x \in \Omega_{3} $,
for a fixed small positive value of $ t $, we have
$$ 
\Phi_{t}(x) - \Phi_{0}(x) \leq \frac{\lambda^{-1 + 1/r}}{C_{2}(t)}.
$$
Therefore $ \| u \|_{\Phi_{0}, \Omega_{3}} \leq c e^{-\lambda^{1/r}/c}
$, which proves Theorem \ref{th:1}.

\section{Proof of Theorem \ref{th:2}}
\renewcommand{\thetheorem}{\thesection.\arabic{theorem}}
\renewcommand{\theproposition}{\thesection.\arabic{proposition}}
\renewcommand{\thelemma}{\thesection.\arabic{lemma}}
\renewcommand{\thedefinition}{\thesection.\arabic{definition}}
\renewcommand{\thecorollary}{\thesection.\arabic{corollary}}
\renewcommand{\theequation}{\thesection.\arabic{equation}}
\renewcommand{\theremark}{\thesection.\arabic{remark}}
\renewcommand{\theconjecture}{\thesection.\arabic{conjecture}}
\setcounter{equation}{0}
\setcounter{theorem}{0}
\setcounter{proposition}{0}
\setcounter{lemma}{0}
\setcounter{corollary}{0}
\setcounter{definition}{0}
\setcounter{remark}{0}

We are going to proceed in the same way as in the previous section,
but using the (order zero) function $ h $ of the assumption. First of
all we deform $ \Lambda_{\Phi_{0}} $ according to \eqref{eq:hj}. Next
we want to deduce a priori estimates for $ P + Q $ where the weight
function $ \Phi_{0} $ is replaced by $ \Phi_{t} $. For the sake of
simplicity let us write \eqref{eq:riprod} as
\begin{equation}
\label{eq:riprod2}
\{ h (x, \xi) , X_{j}(x, \xi) \} = \alpha (x, \xi) X (x, \xi) ,
\end{equation}
where $ X $ denotes a vector whose components are the symbols of the
vector fields and $ \alpha $ is a $ N \times N $ matrix with entries
being real analytic symbols. As before we have $ \Lambda_{\Phi_{t}} =
\exp(i t H_{h}) \Lambda_{\Phi_{0}} $. 

Denote by $ Y_{j}^{t} $, $ j = 1, \ldots, N $, the restriction to $
\Lambda_{\Phi_{t}} $ of $ X_{j} $. We have $ Y_{j}^{t} = X_{j} \circ
\exp(i t H_{h}) $, so that, by our assumptions,
$$ 
\partial_{t} Y^{t} = i \{ h , X\} \circ \exp(i t H_{h}).
$$
We deduce that
$$ 
\left\{
\begin{matrix}
2 \dfrac{\partial Y^{t} }{\partial t}(x, \xi)  =  
  i (\alpha \circ \exp(i t H_{h})) (x, \xi) Y^{t}(x, \xi) \\[16pt]
Y^{t}(x, \xi)_{\big|_{t=0}}  =  X (x, \xi) \hfill 
\end{matrix}
\right . .
$$
From this relation we deduce that there is a $ N \times N $ matrix,
whose entries are real analytic symbols depending real analytically on
the real parameter $ t $, $ b_{t}(x, \xi)  $, such that
\begin{equation}
\label{eq:Yt}
Y^{t}(x, \xi) = b_{t}(x, \xi) X (x, \xi) ,
\end{equation}
and that $ b_{0} = \text{Id}_{N} $. Hence $ b_{t} $ is non singular if $ t $
is small enough.

Denote by $ X^{t} $ the holomorphic extension of $ \re Y^{t} $; since
$ X $ is real on $ \Lambda_{\Phi_{0}} $, using (\ref{eq:Yt}), we
have that
\begin{equation}
\label{eq:trealfields}
X^{t}(x, \xi) = \beta_{t}(x, \xi) X(x, \xi),
\end{equation}
where $ \beta_{t=0}(x, \xi) = \id_{N} $. In particular $
\beta_{t} $ is non singular, provided $ t $ is small.

Then we have
\begin{multline}
\label{eq:Pt}
P (x, \tilde{D}) = \sum_{i,j =1}^{N} X^{t}_{i}(x, \tilde{D})
a^{t}_{ij}(x, \tilde{D}; \lambda)  X^{t}_{j}(x, \tilde{D})  
\\
+ \lambda^{-1} \sum_{j=1}^{N} b^{t}_{j}(x, \tilde{D}; \lambda)
X^{t}_{j}(x, \tilde{D}) + 
\lambda^{-2} c^{t} (x, \tilde{D}; \lambda),
\end{multline}
for suitable analytic pseudodifferential operators $ a^{t}_{ij} $, $
b_{j}^{t} $, $ c^{t} $ of order zero.  

We can apply Theorem \ref{apriorifriendly} and deduce that
\begin{equation*}
\lambda^{\frac{2}{r}} \|u\|_{\Phi_{t}, \Omega_{1}} \leq C \left( \| Pu
  \|_{\Phi_{t}, \Omega} + \lambda^{\alpha} \| u \|_{\Phi_{t},
    \Omega\setminus\Omega_{1}} \right),
\end{equation*}
where $ \Omega_{1}\subset\!\subset \Omega $, $ \alpha $ is a fixed
positive integer and $ P $ denotes the realization on $ \Omega $ of
the given operator $ P $. Let $ Q $ the realization on $ \Omega $
of the real analytic pseudodifferential operator of order
$ \theta  < 2/r$ in the statement of Theorem \ref{th:2}. We have 
\begin{equation}
\label{eq:aprioritperturb}
\lambda^{\frac{2}{r}} \|u\|_{\Phi_{t}, \Omega_{1}} 
\leq C \left( \| \left( P+Q \right) u  \|_{\Phi_{t}, \Omega} +
\| Q u \|_{\Phi_{t}, \Omega}+ \lambda^{\alpha} \| u \|_{\Phi_{t},
    \Omega\setminus\Omega_{1}} \right).
\end{equation}
Let us consider the second term in the right hand side of
the above inequality. We have
\begin{equation*}
\| Q u \|_{\Phi_{t}, \Omega}
\leq C_{1} \lambda^{\theta} \| u \|_{\Phi_{t}, \Omega}
\leq C_{1} \lambda^{\theta}\left(  \| u \|_{\Phi_{t}, \Omega_{1}}
+\| u \|_{\Phi_{t}, \Omega\setminus \Omega_{1}}\right)
\end{equation*}
Since $ \theta < 2/r $ the first term of above inequality is absorbed
on the left hand side of (\ref{eq:aprioritperturb}) provided $ \lambda$
is large enough. Hence we have
\begin{equation}
\label{eq:aprioritperturb-f}
\lambda^{\frac{2}{r}} \|u\|_{\Phi_{t}, \Omega_{1}} 
\leq C \left( \| \left( P+Q \right) u  \|_{\Phi_{t}, \Omega} +
\lambda^{\alpha} \| u \|_{\Phi_{t}, \Omega\setminus\Omega_{1}} \right),
\end{equation}
for a suitable new positive constant $ C $.
 
Assume now that $ (x_{0}, \xi_{0}) \notin WF_{a}((P+Q)u) $. We may
choose $ \Omega $ in such a way that
\begin{equation}
\label{eq:wfpu}
\|  (P+Q) u \|_{\Phi_{0}, \Omega} \leq C e^{-\lambda/C},
\end{equation}
for a suitable positive constant $ C $. From
\begin{equation}
\label{eq:inteq} 
\Phi_{t}(x) = \Phi_{0}(x) + \frac{1}{2} \int_{0}^{t} h\Bigg(x,
\frac{2}{i}\partial_{x}\Phi_{s}(x)\Bigg) ds ,
\end{equation}
using the fact that $ h_{\big | \Lambda_{\Phi_{0}}} \geq 0 $, and
recalling that $ \Lambda_{\Phi_{t}} = \exp (it H_{h})
\Lambda_{\Phi_{0}} $, we deduce that $ h_{\big | \Lambda_{\Phi_{t}}}
\geq 0 $ so that
\begin{equation}
\label{eq:fit>fi0}
\Phi_{t}(x) \geq \Phi_{0}(x), \qquad x \in \Omega.
\end{equation}
Hence, by (\ref{eq:fit>fi0}) and (\ref{eq:wfpu}), 
\begin{equation}
\label{eq:pufit}
\| (P+Q) u \|_{\Phi_{t}, \Omega} \leq C e^{-\lambda/C},
\end{equation}
for a suitable positive constant $ C $.

Let us now estimate the second term in the right hand side of
(\ref{eq:aprioritperturb-f}). 
We point out that
$$ 
h_{\big|\Lambda_{\Phi_{0}}\cap\, \Omega \setminus \Omega_{1}} \geq a > 0. 
$$
It follows, because of (\ref{eq:inteq}), that
\begin{equation}
\label{eq:fit>}
\Phi_{t}(x) \geq \Phi_{0}(x) + c' t, \qquad x \in
\Omega\setminus\Omega_{1}. 
\end{equation}
Then
\begin{eqnarray*}
\| u \|^{2}_{\Phi_{t}, \Omega\setminus\Omega_{1}} & = & 
\int_{\Omega\setminus\Omega_{1}} e^{- 2 \lambda \Phi_{t}(x)} |
u(x)|^{2} L(dx) \\
& \leq  & \int_{\Omega\setminus\Omega_{1}} e^{- 2 \lambda \Phi_{0}(x)
  -2 \lambda c' t} | u(x)|^{2} L(dx) \\
& \leq & C e^{-2\lambda c' t} \lambda^{N}  \\
& \leq & C e^{- \lambda c'' t } , \qquad t > 0. 
\end{eqnarray*}
By (\ref{eq:aprioritperturb-f}) we deduce that $ \| u \|_{\Phi_{t}, \Omega_{1}}
\leq C \exp(- \lambda t /C) $, for a suitable positive constant $ C $.
Let now $ \Omega_{2} \subset\!\subset \Omega_{1} $ be a neighborhood
of $ x_{0} $ such that $ \Phi_{t} \leq \Phi_{0} + t/(2C) $ in $
\Omega_{2} $. We conclude that
$$ 
\| u \|^{2}_{\Phi_{0}, \Omega_{2}} \leq C e^{- \lambda t/C}, \qquad t
> 0.
$$
This proves the theorem.

\setcounter{section}{0}
\renewcommand\thesection{\Alph{section}}
\section{Appendix}
\setcounter{equation}{0}
\setcounter{theorem}{0}
\setcounter{proposition}{0}
\setcounter{lemma}{0}
\setcounter{corollary}{0}
\setcounter{definition}{0}
\setcounter{remark}{0}
\renewcommand{\thetheorem}{\thesection.\arabic{theorem}}
\renewcommand{\theproposition}{\thesection.\arabic{proposition}}
\renewcommand{\thelemma}{\thesection.\arabic{lemma}}
\renewcommand{\thedefinition}{\thesection.\arabic{definition}}
\renewcommand{\thecorollary}{\thesection.\arabic{corollary}}
\renewcommand{\theequation}{\thesection.\arabic{equation}}
\renewcommand{\theremark}{\thesection.\arabic{remark}}

We collect here a few facts concerning the hypoellipticity of
pseudodifferential perturbations of sums of squares. 

Let $ k $ be an integer, $ k \geq 2 $, and consider
$$ 
P (x, D) = D_{1}^{2} + x_{1}^{2(k-1)} D_{2}^{2} , \quad x \in \R^{2}.
$$
Let
$$ 
Q (x, D) = \lambda |D_{2}|^{2/k}.
$$
$ Q $ is microlocally elliptic near points in $ \Char(P) $. Here $
\lambda $ is a constant that we shall choose later. 

Performing a Fourier transform w.r.t. $ x_{2} $, and the dilation 
$$
x_{1} \rightarrow |\xi_{2}|^{- 1/k} x_{1} ,
$$ 
$ P + Q $ becomes, modulo a microlocally  elliptic factor which we can
disregard, 
$$ 
D_{1}^{2} + x_{1}^{2(k-1)} + \lambda.
$$
Let $ \phi_{\lambda}(x_{1}) $ be such that
$$ 
- \phi''_{\lambda} + x_{1}^{2(k-1)} \phi_{\lambda} + \lambda
\phi_{\lambda} = 0. 
$$
This is possible since the above operator, by \cite{berezin}, has a
discrete, positive, simple spectrum, so that, if $ \lambda $ is the
opposite of an eigenvalue, $ \phi_{\lambda} $, the associated
eigenfunction, satisfies the above equation. It is well known
that $ \phi_{\lambda} \in \mathscr{S}(\R) $, i.e. is rapidly
decreasing at infinity.

Consider
\begin{equation}
\label{eq:a1}
u(x) = \int_{0}^{+\infty} e^{i x_{2} \rho} \phi_{\lambda}(x_{1} \rho^{1/k}) (1 +
\rho^{4})^{-1} d\rho. 
\end{equation}
We see immediately that $ (P + Q) u = 0 $. Let us show that $ u \not
\in C^{\infty} $. 

Let us assume first that $ \phi_{\lambda}(0) \neq 0 $. Then
$$ 
u(0, x_{2}) = \phi_{\lambda}(0) \int_{0}^{+\infty} e^{i x_{2}\rho} (1 +
\rho^{4})^{-1} d\rho,
$$
and it is obvious that it cannot be smooth since we cannot take an
arbitrary derivative w.r.t. $ x_{2} $.

If $ \phi_{\lambda} (0) = 0 $ then necessarily $ \phi'_{\lambda}(0)
\neq 0 $. It suffices then to consider 
$$ 
(\partial_{x_{1}}u)(0, x_{2}) = \phi_{\lambda}'(0) \int_{0}^{+\infty} e^{i x_{2}\rho} (1 +
\rho^{4})^{-1} \rho^{1/k}  d\rho,
$$
and argue exactly as in the preceding case. 

This shows that a pseudodifferential perturbation of the same order as
the subellipticity index does not preserve the $ C^{\infty} $
hypoellipticity. Same argument for the analytic hypoellipticity.
 
\bigskip

We also point out that allowing a general pseudodifferential
perturbation of order equal to the subellipticity index may lead to
both a hypoelliptic and a non hypoelliptic operator. 

Consider for instance, microlocally near the
point $ (0, e_{2}) $, $ P $ as above and $ Q =
\lambda | D_{2}|^{2/k} + \mu(x_{2}) |D_{2}|^{\epsilon}$, with $
\epsilon < 2/k $. Then $ P+Q $ can be analytic hypoelliptic, $ G^{s} $
hypoelliptic for some $ s $, or not even $ C^{\infty} $ hypoelliptic,
depending on the analytic function $ \mu $. We do not wish to give any
detail about this since it goes far beyond the scope of the present
note.

\end{document}